\documentclass{amsart}
\usepackage{amsmath} 
\usepackage{amsfonts}
\newtheorem{theorem}{Theorem}[section]

\newtheorem{proposition}[theorem]{Proposition}
\newtheorem{corollary}[theorem]{corollary}
\theoremstyle{definition}
\newtheorem{definition}[theorem]{Definition}

\theoremstyle{remark}
\newtheorem{remark}[theorem]{Remark}
\numberwithin{equation}{section}
\newcommand{\bbR}{{\mathbb{R}}}
\newcommand{\cA}{\mathcal{A}}       
\newcommand{\cB}{\mathcal{B}}

\newcommand{\cD}{\mathcal{D}}
\newcommand{\cH}{\mathcal{H}}
\newcommand{\cP}{\mathcal{P}}
\newcommand{\Aut}{\operatorname{Aut}}
\newcommand{\End}{\operatorname{End}}
\newcommand{\Nm}[1]{\left\| #1\right\|}                        %
\newcommand{\Ad}{\mbox{Ad}\,}

\def\[{\ifmmode{[\hskip -0.2em [}
    \else{\hbox{$[\hskip -0.2em [$}}\fi}
\def\]{\ifmmode{]\hskip -0.2em ]}
    \else{\hbox{$[\hskip -0.2em ]$}}\fi}
\begin{document}
\title[An extension of Stone's Theorem]{Survey on a quantum stochastic extension\\ of Stone's Theorem}
\author{Claus K\"ostler}
\address{Department of Mathematics and Statistics, Queen's University, 
Kingston, Ontario K7L3N6, Canada}
%\curraddr{Department of Mathematics and Statistics,
%Queen's University, Kingston, Ontario K7L3N6, Canada}
\email{koestler@mast.queensu.ca}
\thanks{This paper is an extended version of a talk given at the Conference 
\textit{Advances in Quantum Dynamics} in Mount Holyoke (June 2002).
This work was partially joint with J\"urgen Hellmich and Burkhard K\"ummerer.
We are grateful to Roland Speicher for several fruitful discussions.}

%\author{Author Two}
%\address{Mathematical Research Section, School of Mathematical Sciences,
%Australian National University, Canberra ACT 2601, Australia}

%    General info
\date{November 1, 2002. } 



\begin{abstract}
From K\"ummerer's investigations on stationary Markov processes has emerged
an operator algebraic definition of white noises which captures many examples from classical 
as well as from non-commutative probability. Within non-commutative $L^p$-spaces associated to a white noise,
the role of (non-)commutative L\'evy processes is played by additive cocycles for the white noise shift, 
and moreover, the notion for exponentials of classical L\'evy processes is generalized by 
unitary cocycles. As a main result we report a bijective correspondence between additive and 
unitary cocycles for white noise shifts. If the cocycles are required to be differentiable, 
the presented correspondence reduces to Stone's theorem (for norm continuous unitary groups).

The correspondence needs the development of background results for additive cocycles with 
$L^\infty$-bounded covariance operators: an operator-valued stochastic 
It\^o integration, quadratic variations and non-commutative martingale inequalities as well as 
stochastic differentiation. Related results and recent progress towards the 
case of additive cocycles with unbounded variance operators are reported.  
 
\end{abstract}

\maketitle

\specialsection{Introduction}
\noindent
The developments of non-commutative (or quantum) probability led to many examples of 
non-commutative analogues of Brownian motion or, more generally speaking, L\'evy processes. 
They are realized through creation and annihilation operators on deformed Fock spaces. Motivated
by Classical Probability a natural question arises in quantum probability: 
do non-commutative Markov processes appear as solutions of quantum stochastic differential equations
when the increments are given by non-commutative analogues of L\'evy processes?

An early answer to this question was given in the 80's by Hudson and Partha\-sarathy for the Bosonic
Fock space, as well as Barnett, Streater and Wilde for the Fermionic Fock space,  
and was followed by the work of many others (cf.\ \cite{Meye95a,Part92a} and references therein). 
They succeeded in extending the stochastic It\^o integral, in particular, to creation and annihilation 
operators as increments. Moreover, quantum Markov processes were constructed as solutions of Bosonic or 
Fermionic quantum stochastic differential equations. Beginning in the 90's, 
a third example appeared on the scene with Voiculesu's Free Probability \cite{VDN93a}: free Brownian 
motion which also provides a rich stochastic calculus for the construction of Markov processes \cite{KuSp92a}. 
Meanwhile, the stochastic calculi of these examples have been further developed and of 
particular importance for the present approach are \cite{ApFr86a,BSW83c,BiSp98a,HHKKR}. Investigations on other examples of 
non-commutative analogues of Brownian motion on Fock spaces, e.g., the family of $q$-Brownian motions,
become recently a subject of increasing interest (cf. \cite{BoSp91a,BoSp94a,BKS97a,GuMa02a,Krol02a} and 
references therein).

At about the same time in the 80's, K\"ummerer followed an operator algebraic approach to
stationary Markov processes \cite{Kuem85a}. His investigations showed  
that many of these processes can be realized as so called couplings $\operatorname{Ad} u_t$ to a stationary 
white noise shift $S_t$ \cite{Kuem93a,Kuem96a}. Here the shift $S_t$ is an 
automorphism group acting on a von Neumann algebra $\cA$ and the coupling is given by a unitary 
cocycle $u_t$ for $S_t$. Consequently, the Markovian evolution appears as 
$\cA \ni x \mapsto \operatorname{Ad}u_t \, S_t(x) = u_t^*S_t(x) u_t$. From this approach the question arises whether a unitary 
cocycle $u_t$ for a white noise shift $S_t$ can still be identified as the solution of a certain 
stochastic differential equation $du_t = db_t u_t$. Here the increments $db_t$ are provided
by an additive cocycle $b_t$ (for $S_t$), the operator algebraic generalization of a classical L\'evy 
process. This question was affirmatively answered in \cite{HKK98a}, in the form of a bijective 
correspondence between additive and unitary cocycles. In this survey we present this result in an 
improved form within the framework of non-commutative $L^p$-spaces associated with a trace.
We emphasize that the result relies only on the operator algebraic white noise and its 
cocycles. A priori, it does not refer to any kind of examples on Fock spaces. Nethertheless, 
(deformed) Fock spaces provide rich sources for examples of white noises and additive cocycles
\cite{Kuem96a,BoGu02a}.

For the purpose of this survey we restrict our attention to white noises with non-commutative 
$L^p$-spaces associated with a trace. Appropriate generalizations of the results to white noises
on von Neumann algebras of type III are available by passing to Haagerup's $L^p$-spaces. 

After introducing necessary notations and properties of non-commutative probability as well as
non-commutative $L^p$-spaces in Section \ref{preliminaries}, we proceed in Section \ref{whitenoise} 
with the crucial definition of operator-valued white noise. 
The presented definition of white noise includes a suitable notion of 
independence, similar to the structure of commuting squares in subfactor theory \cite{GHJ89a}. 
Moreover, additive and unitary cocycles for white noise shifts are defined. They will play the part 
of operator-valued L\'evy processes  resp.\ their unitary exponentials. Finally, the introduced 
framework is justified by a result of K\"ummerer \cite{Kuem93a} which states that many 
non-commutative Markov processes come from cocycle pertubations of white noises.  

The main result is contained in Section \ref{mainresult}. In the presence of an operator-valued 
white noise, it states a bijective correspondence between additive and unitary cocycles for the 
white noise shift: Briefly formulating, $u_t$ is a unitary cocycle iff $b_t$ is an additive cocycle 
satisfying the structure equation $db_t^*db_t +db_t+db_t^*=0$. 
As a Corollary, Stone's Theorem appears when the cocycles are further required to be differentiable.  
Another familiar Corollary is the construction of stationary Markov processes from the pertubation of 
the white noise shift with a unitary cocycle. Due to the correspondence, the generator of the associated 
Markovian semigroup can easily be expressed in terms of the corresponding additive cocycle. A first, 
preliminary version of such a bijective correspondence is contained in \cite{HKK98a} using the language of 
Hilbert modules. The present approach, using non-commutative $L^p$-spaces, enables us to prove the existence 
of all moments of an additive cocycle whenever it comes from the correspondence \cite{Koes00a}.

To establish the correspondence, one has to construct a unitary cocycle from an additive one 
and vice versa. The first direction needs stochastic integration for additive 
cocycles of white noise shifts \cite{Koes00a} and constructs the unitary cocycle as the solution of
a certain stochastic differential equation. This approach combines the structure of an 
operator-valued white noise with general stochastic integration on non-commutative $L^2$-spaces 
as can be found in \cite{BSW87a}. We will only briefly review the related concepts in Section 
\ref{integration}, as far as it necessary for the formulation of the main result. 

Section \ref{variation} is devoted to quadratic variations of additive cocycles \cite{Koes02b}, 
which give rise to the non-commutative analogues of the famous It\^o corrections $db_t^*db_t$ 
in classical stochastic analysis. The stated results provide a rigorous formulation of the structure equation 
$db^*_t db_t + db +db_t^*=0$.  
Furthermore, we present estimates on the growth of higher moments of centred additive cocycles 
\cite{Koes02b}.
This gives applications to non-commutative martingale inequalities as they have been established in 
the work of Junge, Pisier and Xu for discretely indexed martingales in \cite{PiXu97a,JuXu01a}. 
   
In Section \ref{differentiation} we present the construction of additive cocycles from unitary
cocycles. We call this procedure stochastic differentiation and its idea is already present in 
the work of Skorokhod on operator-valued stochastic semigroups in Classical Probability \cite{Skor82a}. 
Independently, the idea reappeared for scalar-valued white noises (cf.~\cite{Kuem96a}).
We will present results on the existence of all moments 
of the stochastic derivative and the general case which leads to additive cocycles with 
$\tau$-measurable covariance operators. 
 
Finally, we investigate in Section \ref{flucdiss} the structure equation $db_t^*db_t + db_t + db^*_t=0$
and introduce flows of additive cocycles. Consequently, the structure equation
is identified as a specific case of a non-commutative fluctation-dissipation theorem. This provides 
sufficient information to control the unbounded generator of the Markovian semigroup.

\specialsection{Preliminaries}\label{preliminaries} 
\noindent
In the following we introduce non-commutative probability spaces, their morphisms,
filtrations and notions of non-commutative processes which we will use throughout this 
paper. For this survey we will limit our considerations to non-commutative $L^p$-spaces 
associated with a trace, but we emphasize that the presented setting carries over to 
Haagerup's $L^p$-spaces \cite{Haag79a,Terp81a}, after some suitable modifications.  

$\cA$ will always denote a finite von Neumann algebra with
norm separable predual $\cA_*$. A faithful normal tracial state on $\cA$ will be 
denoted by $\tau$. For $1\le p \le \infty$ the non-commutative $L^p$-spaces 
$L^p(\cA,\tau)$ are defined by the completion of $\cA$ in the norm 
$\Nm{x}_p = (\tau(|x|^p))^{1/p}$, $x\in \cA$, where $|x| = (x^*x)^{1/2}$.  
We recall that $L^\infty(\cA,\tau)$ is just $\cA$ with the usual operator norm and 
that $L^1(\cA,\tau)$ is isometrically isomorphic to $\cA_*$. Note that in the 
presented setting, a densely-defined closed operator affiliated with $\cA$ is already
$\tau$-measurable. We will denote all $\tau$-measurable operators by $L^0(\cA,\tau)$. Moreover, 
the trace $\tau$ extends to a positive tracial functional
on the positive part $|x|$ of all $\tau$-measureable operators $x$ and its 
extension will still be denoted by $\tau$. Note that $x\in L^p(\cA,\tau)$ iff
$\tau(|x|^p) < \infty$ ($1\le p<\infty$). For the definitions of $\tau$-measurability
and further details on non-commutative $L^p$-spaces we refer the reader to \cite{PiXu02a} 
and the references therein. 

In the sequel the pair $(\cA,\tau)$ and its associated non-commutative $L^p$-spaces
will be called (non-commutative) probability spaces. A morphism $T: (\cA,\tau) \to 
(\tilde{\cA},\tilde{\tau})$ between two probability spaces
is a completely positive operator such that $\tilde{\tau}\circ T = \tau$
and $T(1_\cA)=1_{\tilde{\cA}}$. Note that a morphism $T$ extends to a contraction 
from $L^p(\cA,\tau)$ to $L^p(\tilde{\cA},\tilde{\tau})$ ($1\le p <\infty$), denoted 
also by $T$. Endomorphism and automorphism of a probability space are accordingly 
understood and denoted by $\End(\cA,\tau)$ resp.\ $\Aut(\cA,\tau)$. 
If $\cB \subset \cA$ is a von Neumann subalgebra it is well known that there exists a 
unique conditional expectation $E\colon \cA \to \cB$ as endomorphism. It extends
to a contractive projection from $L^p(\cA,\tau)$ onto $L^p(\cB,\tau)$,
which is still called a conditional expectation. Moreover, $L^p(\cB,\tau)$ is always 
naturally identified with a subspace of $L^p(\cA,\tau)$. 

A family of von Neumann subalgebras $(\cA_I)_I\subset \cA$, indexed by 
'intervals' $I=[s,t]$ with $-\infty \le s\le t \le \infty$, is called a 
{\em filtration} of $(\cA,\tau)$ if $\cA_{I}$ and $\cA_J$ generates $\cA_{K}$ 
whenever $I\cup J = K$.  The filtration $(\cA_I)_I$ is called {\em minimal}
if $\cA_\bbR = \cA$ and {\em lower} (resp.\ {\em upper}) {\em continuous} if 
$\bigcup_{I \subset J^\circ} A_I$ generates $A_J$ 
(resp.~$\bigcap_{I^\circ \supset I} A_J =\cA_I$) as von Neumann algebra. Here we 
used the convention $[s,t]^\circ = (s,t)$ for $s<t$ and $ [s,s]^\circ = \emptyset$. 
Let $E_I$ denote the conditional expectation from $\cA$ onto $\cA_I$. The filtration 
$(\cA_I)_I$ is 
called {\em $\cA_0$-Markov filtration} if $E_{(-\infty,0]} E_{[0,\infty)} = E_0$.  
Note that the filtration and its related conditional expectations extend to 
non-commutative $L^p$-spaces for $1\le p< \infty$ and will be called the same.

Let $\cA_0\subset \cA$ be a distinguished von Neumann subalgebra and let 
$T\equiv (T_t)_{t \in \bbR}\subset \Aut(\cA,\tau)$ be a pointwise weakly* continuous  
group. The tuple $(\cA,\tau,T;\cA_0)$ is called a {\em $\cA_0$-valued (stationary)
stochastic process}. The group $T$ and $\cA_0$ induce a continuous
filtration $(\cA_I)_{I}$ on the probability space $(\cA,\tau)$ through the von Neumann
subalgebras $\cA_I$ generated by $\bigcup_{t \in I}T_t(\cA_0)$. 
If $(\cA_I)_I$ is a $\cA_0$-Markov filtration then $(\cA,\tau,T;\cA_0)$ is called
a {(stationary) $\cA_0$-Markov process}. Note that in this case 
the compressions $R = E_0T E_0 \in \End(\cA_0,\tau)$ define a semigroup of 
contractions on $\cA_0$. Again, these notions extend canonically to 
non-commutative $L^p$-spaces.

Finally, following the approach of K\"ummerer \cite{Kuem93a}, we formulate the non-commutative 
counterpart to operator-valued classical stationary generalized processes. We start with a filtration
$(\cA_I)_I \subset \cA$ (which contains $\cA_0$) and a pointwise weakly* continuous 
group $T\subset \Aut(\cA,\tau)$ such that both are compatible, i.e., 
$T_t(\cA_{[r,s]})= \cA_{[r+t,s+t]}$ for any $-\infty \le r\le s\le \infty$ and 
$t \in \bbR$. Then $(\cA,\tau,T;(\cA_I))$ is called  {\em $\cA_0$-valued 
generalized (stationary) process}. By the continuity of the automorphism group
its filtration is automatically lower continuous. Upper continuity follows for 
the families $(\cA_{(-\infty,t]})_{t\in \bbR}$ and $(\cA_{[t,\infty)})_{t\in \bbR}$. 
The generalized stationary process $(\cA,\tau,T;(\cA_I)_I)$ is called  
{\em $\cA_0$-valued noise} if the filtration $(\cA_I)_I$ is minimal and $\cA_0$ is
the fixed point algebra of $T$. This notions extend as before to non-commutative $L^p$-spaces. 
We will use the same notation for them.

\specialsection{White noises and their cocycles}\label{whitenoise}
\noindent
We will present the non-commutative version of an 
operator-valued classical white noise. 
Recall that $\cA$ denotes a finite von Neumann algebra which is equipped with
a faithful normal tracial state $\tau$ and a minimal filtration $(\cA_I)_I$. 
We remind that $E_I$ denotes the conditional expectation from $\cA$ onto $\cA_I$.   
Finally, $S$ denotes an automorphism group which is compatible with the 
given filtration and has the fixed point algebra $\cA_0$. For the precise 
definition of a noise we refer to the preliminaries. 
\begin{definition}\label{noise}
 A noise $(\cA,\tau,S;(\cA_I)_I)$ is called a 
{\em($\cA_0$-valued) white noise} if $E_I \circ E_J = E_0$ whenever 
$I^\circ\cap J^\circ =\emptyset$. The von Neumann subalgebras $\cA_I$ 
and $\cA_J$ of $\cA$ are called {\em independent (over $\cA_0$)} if 
$I^\circ\cap J^\circ= \emptyset$ implies $E_I \circ E_J = E_0$.  
\end{definition}
If the von Neumann algebra $\cA$ is commutative, one gets back classical examples 
of white noises (cf.\ \cite{Hida80a} for the example of Gaussian white noise and \cite{HKK98a} for its
reformulation to a white noise in our sense).  
Crucial for the noise to be 'white' is the notion of independence which parallels 
commuting squares in subfactor theory \cite{GHJ89a}. Moreover, it
reduces to the classical notion of stochastic independence whenever
the underlying von Neumann algebra is commutative and the noise scalar-valued. 

Throughout the paper we will always assume that an $\cA_0$-valued white noise 
$(\cA,\tau, S; (A_I)_I)$ and its canonical extensions to
non-commutative $L^p$-spaces ($1\le p \le \infty$) are given.
For brevity, 
the automorphism group of the white noise and its extensions to the 
associated non-commutative
$L^p$-spaces will just be called {\em shift} $S$. 
\begin{remark}
(i) The notion of white noise does not refer to Fock spaces. However, many 
examples of non-commutative white noises come along with generalized 
Brownian motions on deformed Fock spaces
\cite{BoSp94a,BKS97a,HKK98a,Koes00a,HHKKR,GuMa02a,BoGu02a,Krol02a,Hell02a,HKK02a}.
Other examples come from classical compound
processes which trigger non-commutative Bernoulli shifts
\cite{Kuem87a,Rupp91a,Hell02a,HKK02a}.

(ii) All examples of white noises constructed from so called white noise 
functors satisfy $E_I E_J = E_{I\cap J}$ \cite{Kuem96a,GuMa02a}. This implies
independence as stated in definition \ref{noise}. The converse is an open problem.
\end{remark} 
We define the analogue of an operator-valued class of classical L\'evy processes. 
\begin{definition}
Let $1 \le p \le \infty$. An {\em additive $L^p$-cocycle} $b$ for the shift 
$S$ is a strongly continuous family $(b_t)_{t\ge 0}$ in $L^p(\cA,\tau)$ such that
\begin{enumerate}
\item[(i)] $b_t \in L^p(\cA_{[0,t]},\tau)$ (adaptedness);
\item[(ii)] $b_{s+t} = S_t(b_s) + b_t$ for all $s,t \ge 0$ (cocycle identity). 
\end{enumerate}
If $E_0(b_t)=0$ for any $t\ge 0$ then $b$ is called {\em centred}. If $E_0(b_t)=b_t$
then $b$ is called {\em drift}. $E_0(b_1^*b_1)$ is the {\em covariance (operator)} of $b$. 
\end{definition}
Frequently, an additive $L^p$-cocycle will also be called 
a {\em $L^p$-L\'evy process}. 
A centred additive $L^p$-cocycle $b$ is automatically
continuous since it is a martingale with respect to the continuous filtration 
$(\cA_{(-\infty,t]})_{t>0}$. Moreover, the continuity of a centred additive 
$L^p$-cocycle $c$ implies the identity
$$E_0(c_t^* x c_t) = E_0(c_1^*x c_1) t$$
for all $t>0$, whenever $x \in L^q(\cA_0,\tau)$ and $1 \ge  2/p + 1/q$. The equality features
that additive cocycles lead to operator-valued Lebesgue measures. This
observation is a cornerstone for a non-commutative version of the It\^o integral
(cf.~ Section \ref{integration}).

Note that the continuity assumption is only needed for a drift and
insures that a drift is always of the form $(td )_{t\ge 0}$ for some 
$d \in L^p(\cA_0,\tau)$. We remark that an additive cocycle $b$
uniquely decomposes
into a centred additive cocycle $c$ and a drift $d$ such 
that $b_t = c_t + td $ for
 any $t\ge 0$. Our main result will require an additive $L^2$-cocycle
$b$ with a symmetric covariance
$E_0(b^*_1b_1 + b_1b_1^*) \in L^\infty(\cA_0,\tau)$.

Let us now present the non-commutative analogue of 
exponentials of L\'evy processes on the unit circle.   
\begin{definition}
A {\em unitary cocycle} $u$ for the shift $S$ is a weakly* continuous 
family of unitary operators $(u_t)_{t\ge 0}\in L^\infty(\cA,\tau)$ 
with the following properties:
\begin{enumerate}
\item[(i)] $u_t \in \cA_{[0,t]}$ (adaptedness);
\item[(ii)] $u_{s+t} = S_t(u_s) u_t$ for all $s,t \ge 0$ (cocycle identity). 
\end{enumerate}
\end{definition}
An immediate consequence of (i) and (ii) is that 
$t\mapsto E_0(u_t) \subset L^\infty(\cA_0,\tau)$ defines a strongly continuous 
semigroup of contractions on $L^p(\cA_0,\tau)$ ($1\le p\le \infty$). For our main 
result we will assume that $E_0(u_t)$ is uniformly continuous. Note that 
$E_0(u)$ is uniformly continuous iff its generator $d$ is an operator in 
$L^\infty(\cA_0,\tau)$. The more general case will be further considered in 
Section \ref{differentiation}.

Stationary Markov processes appear now as follows. Let us extend the unitary cocycle
$u$ to negative times by defining $u_t = S_t u_{-t}^*$ for $t<0$ and let $\Ad u_t(x)
= u_t^*x u_t$ for any $x \in L^\infty(\cA,\tau)$. 
\begin{proposition}
Let $(\cA,\tau, S,(\cA_I)_I)$ be an $\cA_0$-valued white noise with unitary cocycle
$u$. Then $(\cA, \tau, \operatorname{Ad}u \, S; \cA_0)$ is a stationary Markov process with values in 
$\cA_0$.
\end{proposition}
Such a representation of a Markov process is called a {\em coupling to
white noise} \cite{Kuem93a}.
Finally, let us state a result from \cite{Kuem93a} which establishes the converse and
justifies the definition of a white noise. $M_n$ will denote the 
$n \times n$-matrices which are identified with their isomorphic
embedding in $\cA$. 

\begin{theorem}
Let $(\cA, \tau, T; M_n)$ be a stationary Markov process. 
If there exists a minimal projection $e \in M_n$ which is
invariant under $T$, then the Markov process is a coupling to a
$M_n$-valued white noise $(\cA,\tau, S, (\cA_I)_I)$.
\end{theorem}

This result indicates how intimately white noises and Markov processes are
related and motivated investigations which led to the main result 
stated below.

\specialsection{The main result}\label{mainresult}
\noindent
We are now in the position to present our main result. From a
stochastic point of view it states that many Markov processes 
are driven by L\'evy processes. From the perspective of
functional analysis it is an extension of Stone's theorem to cocycles
of unitary groups.

\begin{theorem}\label{stone}
For an $\cA_0$-valued white noise $(\cA, \tau, S; (\cA_I)_I)$ there exists
a bijective correspondence between: 
\begin{itemize}
\item[(a)] $u \subset L^\infty(\cA,\tau)$ is a unitary adapted cocycle 
for $S$ such that its semigroup $E_0(u)$ is norm continuous in 
$L^\infty(\cA_0,\tau)$.
\item[(b)] $b\subset L^2(\cA,\tau)$ is an additive adapted cocycle for $S$ with 
symmetric covariance
$E_0(b^*b +bb^*) \subset L^\infty(\cA_0,\tau)$ that satisfies the 
structure equation
$$\[b,b\]+ b^*+b =0.$$
\end{itemize} 
The correspondence from \textup{(a)} to \textup{(b)} is given by the stochastic differentiation
$$
b_t = \lim_{N\to \infty} \sum_{j=0}^{N-1}S_{jt/N}(u_{t/N}-1)
$$
and from \textup{(b)} to \textup{(a)} by the stochastic differential equation
$$
u_t = 1 + \int_{s=0}^t db_s u_s.
$$
\end{theorem}
Here, $\[b,b\]$ denotes the quadratic variation of the additive cocycle $b$
(cf.~Section \ref{variation}), the definition and background of the stochastic differential equation
is presented in Section \ref{integration}. 
The main result reduces to Stone's theorem (for bounded generators) if we require that
the cocycles are differentiable. 

\begin{corollary}
If the unitary cocycle $u$ in \textup{(a)} or the additive cocycle $b$ in \textup{(b)} is differentiable
then $u_t = \exp{(ith)}$ and $b_t= ith$ for some selfadjoint operator $h$ in 
$L^\infty(\cA_0,\tau)$.  
\end{corollary} 
\begin{proof}
We will show that a differentiable unitary cocycle $u$ lies in $\cA_0$, the fixed point algebra
of the shift $S$. Since $u_t-1 \in \cA_{(-\infty,t]} \cap \cA_{[0,\infty)}$ for any $t>0$
we conclude $\lim_{t\to 0}\frac1t(u_t-1) \in \bigcap_{t>0} \cA_{(-\infty,t]}$. 
The continuity from above of the filtration $(\cA_{(-\infty,t]})_{t \in \bbR}$ implies
$\lim_{t\to 0}\frac1t(u_t-1) \in \cA_{(-\infty,0]}$. But it holds also 
$\lim_{t\to 0}\frac1t(u_t-1) \in \cA_{[0,\infty)}$. Thus the derivative $u^\prime_0$ has to be 
in $\cA_0$. Furthermore, the cocycle identity implies
$u^\prime_t = S_t(u^\prime_0) u_t = u^\prime_0 u_t$ and therefore 
$u_t = 1 + \int_0^t u^\prime_s ds = 1 + \int_0^t u^\prime_0 u_s ds$ is in $\cA_0$ for any 
$t \ge 0$. Since the solution $u_t = \exp{(u^\prime_0 t)}$ is unitary, we conclude
$u^\prime_0 = i h$ for some selfadjoint $h \in \cA_0$. 

On the other hand,  the additive cocycle uniquely decomposes  as $b_t = c_t + ih t$, 
where $c$ is a centred additive cocycle and $h\in \cA_0$. 
From $E_0(|c_t|^2) = E_0(|c_1|^2) t$ and the differentiability 
of $c_t$ we conclude $c_1=0$ and from this $c_t=0$. The selfadjointness of $h$ follows
immediately from the structure equation. 
\end{proof}

By the means of our main results stationary non-commutative Markov processes
can easily be constructed from additive cocycles for a white noise shift.

\begin{corollary}
Let $b$ be an additive cocycle (with centred part $c$ and drift $d$) 
which satisfies the conditions stated in \textup{(a)} and let
$u$ be the corresponding unitary cocycle. Then 
$(\cA,\tau, \operatorname{Ad}u \, S;\cA_0)$ is a stationary Markov process with values
in $\cA_0$. Moreover, the generator $L$ of the contractive 
semigroup $R = E_0 (\operatorname{Ad}u) E_0$ has the form
$$ 
L (x) = E_0(c_1^*xc_1) + d^*x + x d, \quad x \in \cA_0.
$$
\end{corollary}
As already stated, additive cocycles are non-commutative generalizations
of classical L\'evy processes. In our main result we assumed
that second moments of the additive cocycle exist. But the structure 
equation (in $L^1(\cA,\tau)$) is very restrictive as the following result shows:

\begin{theorem}
If an additive $L^2$- cocycle $b$ satisfies the condition \textup{(b)} in
\textup{Theorem} \textup{\ref{stone}}, 
then all moments of $b$ exist, i.e., $\tau(|b|^p) < \infty $ for any $1 \le p <\infty$. 
\end{theorem}
The proof uses non-commutative martingale inequalities 
and was obtained in \cite{Koes00a} by approximations starting from the 
non-commutative Burk\-holder-Gundy inequalities in
\cite{PiXu97a}. Meanwhile an alternative proof
is given in \cite{Koes02c} which starts from the non-commutative Burkholder/Rosenthal 
inequalities in \cite{JuXu01a}.

We conjecture that the assumptions in our main result on the norm continuity of the 
semigroup $E_0(u)$ and the boundedness of the symmetric variance  $E_0(b^*b +bb^*)$ 
can be tremendiously relaxed. This is indicated by results which are used in the 
proof of the bijective correspondence between additive and unitary cocycles.

\specialsection{Stochastic integration}\label{integration} 
In this section we review non-commutative stochastic integrals and 
differential equations of It\^o type as far as they are needed in 
theorem \ref{stone}.

Let an $\cA_0$-valued white noise $(\cA, \tau, S;(\cA_I)_I)$, its canonical 
extensions to $L^p$-spaces ($1\le p <\infty$) and an additive $L^p$-cocycle for 
the shift $S$ be given. The following 
observation is the key to build It\^o integrals on non-commutative $L^2$-spaces
of white noises. 

\begin{proposition}
Let $x\in L^2(\cA_I,\tau)$ and $y \in L^2(\cA_J,\tau)$ be independent, i.e., 
$I^\circ \cap J^\circ = \emptyset$. 
Then $E_0(x^*x) \in L^\infty(\cA_0,\tau)$ implies $xy\in L^2(\cA,\tau)$. Moreover,
$$
\Nm{xy}_2^{} = \|(E_0(x^*x))^{1/2}y\|_2^{} \text{ and } E_0(y^*x^*xy) = E_0(y^*E_0(x^*x)y).
$$  
\end{proposition}
With this result we are prepared to define stochastic integrals for additive cocycles.
Let $\{0 =s_0 < s_1 <\ldots <s_n =t\}$ be a subdivision of the interval $[0,t]$ and let
the simple stochastic process 
$y = \sum_{j=0}^{n-1} y_{s_j} \chi_{[s_{j}, s_{j+1})} \subset 
L^2(\cA, \tau)$ be adapted to the filtration $(\cA_{[0,s]})_{s\ge0}$, 
i.e., $y_{s} \in L^2(\cA_{[0,s]},\tau)$ for any $s \ge 0$. Then the (left) It\^o integral of $x$ 
w.r.t.\ the centred additive $L^2$-cocycle $b$ is defined to be
$$
\int_0^t db_s y_s = \sum_{j=0}^{n-1} (b_{s_{j+1}}-b_{s_j}) y_{s_j}.
$$ 
This leads to the following It\^o isometry resp.\ identity.
\begin{proposition} Let $y \subset L^2(\cA,\tau)$ be a simple adapted process, then
\begin{eqnarray*}
\Nm{\int_0^t db_s y_s}_2^2 &=& \int_0^t \Nm{(E_0(b_1^* b_1))^{1/2} x_s}_2^2 ds,\\
E_0( |\int_0^t db_s y_s|^2) &=& \int_0^t E_0(x^*_s E_0(b^*_1 b_1)x_s) ds.
\end{eqnarray*}
\end{proposition}
The first identity allows the extension of the It\^o integrals to the Hilbert space
$\cH([0,t])$ of adapted square integrable $L^2(\cA,\tau)$-valued functions on $[0,t]$. Note that the 
assumed boundedness of the covariance $E_0(b^*_1b_1)$ implies that the It\^o integral
defines a bounded linear map on $\cH([0,t])$. The definition of the integral $\int_0^t db_s y_s$
for a non-centred additive cocycle $b$ follows easily by decomposing $b$ 
into its centred part and its drift. Consequently, one can prove the following existence
and uniqueness theorem for solutions of stochastic differential equations.
\begin{theorem}\label{sde}
Let $b$ be a (non-)centred additive $L^2$-cocycle with  covariance
$E_0(b_1^*b_1) \in L^\infty(\cA_0,\tau)$. Then the stochastic differential equation
$$
y_t = 1 + \int_{0}^t db_s y_s 
$$
has a unique solution $x$ in $\cH([0,t])$. 
\end{theorem}
The proof is a typical application of Banach's fixed point theorem
and uses only the boundedness of the covariance operator and the Lipschitz-continuity 
of the identity map on $\cH([0,t])$. Obviously, the theorem extends to a larger class of
Lipschitz continuous maps than the identity. 
We note that the definition of (right) integrals $\int_{0}^t y_s db_s$ is along the 
same lines as for the (left) integral and leads to similar results if one requires 
$E_0(b_1 b^*_1) \in L^\infty(\cA_0,\tau)$. 
We remark that for additive $L^p$-cocycles with covariance
in $L^q(\cA_0, \tau)$, where $q<\infty$, the Banach fixed point theorem no longer holds in 
$\cH([0,t])$. 

Finally, the following result establishes the construction of a unitary cocycle from an additive 
one.
\begin{theorem}
The solution of the stochastic differential equation in \textup{Theorem \ref{sde}} is a unitary cocycle if
the additive $L^2$-cocycle $b$ satisfies the structure equation $\[b,b\] + b + b^* =0$. 
\end{theorem}
  
\specialsection{Quadratic variations}\label{variation}
\noindent
Quadratic variations of L\'evy processes are fundamental for the development
of stochastic calculi in Classical Probability. They give rise to the
famous It\^o corrections. This section is devoted to
quadratic variations of non-commutative L\'evy processes as they appear in the white
noise approach. 

Let $\cP= \{0=t_0 < t_1 < \ldots < t_{n_\cP}=t\}$ denote a subdivision of the interval
$[0,t]$ with mesh $|\cP|= \max_j{\{t_{j+1}-t_j}\}$. 

 \begin{theorem}
Let $b$ be an additive adapted cocycle for $S$ in $L^p(\cA,\tau)$ where $2\le p < \infty$.
Then its quadratic variation
$$
\[b,b\]_t = L^1\text{-}\lim_{|\cP| \to 0} \sum_{j=0}^{n_{\cP}}
|b_{t_{j+1}}-b_{t_j}|^2 
$$
exists for any $t>0$ and defines the (non-centred) additive cocycle $\[b,b\]$ in $L^{p/2}(\cA,\tau)$. 
\end{theorem}
If $b$ and $c$ are two additive cocycles in $L^p(\cA,\tau)$ ($2\le p <\infty$) then
the mutual quadratic variation $\[b,c\]$ is defined by polarisation. The quadratic
variation extends easily to a non-centred additive $L^2$-cocycle $b$ and it holds
$\[b,b\] = \[c,c\]$, where $c$ denotes the centred part of $b$. As
usual, the quadratic
variation of a drift as well as its mutal quadratic 
 variation with an additive cocycle vanishes.
In addition, the identity
$$
\[b,c\]_t = b^*_tc_t - \int_0^tb^*_s dc_s - \int_0^t db^*_s c_s, \qquad t\ge 0,
$$
holds for additive $L^4$-cocycles $b$ and $c$. 

Centred additive cocycles are $L^p$-continuous martingales w.r.t. the filtration 
$(\cA_{(-\infty, t]})_{t\ge 0}$. Moreover, by the above result, their quadratic
variation exists. Starting from the non-commutative Burkholder-Gundy
inequalities of Pisier and Xu \cite{PiXu97a} for martingales indexed
by discrete time, we established \cite{Koes02b}: 
 
\begin{theorem}\label{bg} 
Let $2 \le p <\infty$. Then there exists constants $\alpha_p$ and
$\beta_p$ only depending on $p$ such that for any centred additive
$L^p$-cocycle $b$ and any $t \ge 0$
$$
\alpha_p \Nm{b_t}_{\cH_p(\cA)}\le \Nm{b_t}_p \le \beta_p \Nm{b_t}_{\cH_p(\cA)}
$$
where $\Nm{b_t}_{\cH_p(\cA)} = \max\{\|\[ b,b\]_t^{1/2}\|_p, 
                                  \|\[ b^*,b^*\]_t^{1/2}\|_p\}$.
\end{theorem}
Note that $\alpha_p$ and $\beta_p$ are the constants which appear in the
Burkholder-Gundy inequalities for discretely indexed martingales 
as presented in \cite{PiXu97a} or \cite{JuXu01a}.

Explicit bounds of growth for centred additive cocycles are found in  
\cite{Koes02b} by an application of non-commutative Burkholder/Rosen\-thal inequalities 
\cite{JuXu01a}. 

\begin{theorem}
Let $2 \le p < \infty$. There exist positive constants $\tilde{\alpha}_p$ and
$\tilde{\beta}_p$ such that for every centred L\'evy $L^p$-process $b$ and
all $t \ge 0$ 
$$
    \tilde{\alpha}_p\max\{\mu_p^- t^{1/p}, \lambda_p t^{1/2} \} 
\le \Nm{b_t}_p
\le \tilde{\beta}_p\max\{\mu_p^+ t^{1/p}, \lambda_p t^{1/2} \},
$$
where the constants are defined to be
$$\lambda_p = \max\{ \|(E_0(b_1^*b_1))^{1/2}\|_p,
                   \|(E_0(b_1b_1^*))^{1/2}\|_p\}$$ and 
$$
\mu_p^- = \liminf_{s \searrow 0} \frac{1}{s^{1/p}}\Nm{b_s}_p, \quad    
\mu_p^+ = \limsup_{s \searrow 0} \frac{1}{s^{1/p}}\Nm{b_s}_p < \infty.
$$
\end{theorem}
The constants $\lambda_p$ describe norms of the covariance operator of the 
additive cocycle. If the $p$-th absolute moment
$M_p(t):=\tau(|b_t|^p)$ is right differentiable at $t=0$
then $\mu_p^- = \mu_p^+ = (M_p^\prime(0))^{1/p}$. In the case $p=2,4$
this is a matter of fact, nevertheless we strongly conjecture 
that this property holds for any $2 \le p < \infty$.

\specialsection{Stochastic differentiation}\label{differentiation}
\noindent
This section presents results which are related to the construction
of additive cocycles from unitary cocycles for the shift $S$ of
a given white noise. We call this procedure {\em stochastic differentiation}
because it reduces to the usual differentiation as soon as we require the 
unitary cocycle to be differentiable. 

In the following we always assume that an $\cA_0$-valued white noise
$(\cA,\tau,S;(\cA_I)_I)$ and its extensions to (non-)commutative $L^p$-spaces
($1\le p < \infty)$ are given. We first present the result which is relevant
for Theorem \ref{stone}, i.e, the construction of additive cocycles with
a covariance operator in $L^\infty(\cA_0,\tau)$. 
 
\begin{theorem}
Let $u$ be a unitary cocycle for $S$ such that the semigroup $E_0(u)$ has
the generator $d\in \cA_0$. Let $1 \le p < \infty$. Then 
$$
b_t = L^p\text{-}\lim_{N\to \infty} \sum_{j=0}^{N-1}S_{jt/N}(u_{t/N}-1) 
$$
exists for any $t>0$ and defines an additive cocycle $b$ for $S$ 
in $\bigcap_{p < \infty} L^p(\cA,\tau)$. 
\end{theorem}
A first, direct proof of this result is contained in \cite{Koes00a}
and motivated further investigations on quadratic variations of
additive cocycles. The proof is quite technical and its strategy 
is as follows. In a first step one establishes the convergence
of the ansatz in the norm topology of $L^2(\cA,\tau)$. Next one 
considers the case $2<p<\infty$. Since the ansatz is already
convergent in $L^2(\cA,\tau)$, it is sufficient to find
a uniform bound of the approximating sequence in $L^p(\cA,\tau)$. 
This already guarantees that the limit is in $L^p(\cA,\tau)$ and 
that the ansatz is convergent in the norm topology on $L^q(\cA,\tau)$ 
for $q<p$. Thus one is left with the problem to find a uniform bound 
of the approximating sequence in $L^p(\cA,\tau)$. In \cite{Koes00a} 
this bound was established by tedious approximations which
are based on the non-commutative Burkholder-Gundy inequalities of
Pisier and Xu \cite{PiXu97a}. This concrete example anticipated in
parts more general work of Junge and Xu on non-commutative 
Burkholder-Rosenthal inequalities \cite{JuXu01a}. Meanwhile, a simpler 
proof of the theorem is available in \cite{Koes02c} which is based on 
the results in \cite{JuXu01a}.  

In the following we investigate the case that the semigroup $E_0(u)$
of the unitary cocycle is no longer continuous in the uniform norm 
topology on $L^\infty(\cA_0,\tau)$. This situation leads to the 
construction of additive cocycles with unbounded covariance
operators. Before stating the results let us apply to our case some general 
theory of strongly continuous contractive semigroups on Banach spaces
\cite{Davi80a}.

From $L^\infty(\cA_0,\tau) \subset \cB(L^p(\cA_0,\tau))$ for any 
$1 \le p \le \infty$ we conclude that the generator of $E_0(u)$ is 
a $\tau$-measurable operator $d \in L^0(\cA_0,\tau)$ which acts by 
left-multiplication on the domain $\cD^{p}(d) \subset
L^p(\cA_0,\tau)$. Note that $\cD^{(p)}(d) \supset \cD^{(q)}(d) \supset \cD^{(\infty)}(d)$ 
for $1\le p <q \le \infty$ and, in addition, that
$\cD^{\infty}(d)$ is dense in $L^\infty(\cA_0,\tau)$. Consequently $\cD^{\infty}(d)$ 
is dense in $L^2(\cA_0,\tau)$.  
Let us remind that $E_0(u)$ is uniformly continuous 
on $L^p(\cA_0,\tau)$ iff $d \in L^\infty(\cA_0,\tau)$. Finally, recall
that the vector $a \in L^p(\cA_0,\tau)$ is said to be an analytic vector 
for the generator $d$ of the semigroup $E_0(u)$ on $L^p(\cA,\tau)$ if 
$\sum_{n=0}^\infty \frac{1}{n!}\Nm{d^na}_p <\infty$.

The following result is already contained in \cite{Koes00a}:
\begin{theorem}
Let $u$ be a unitary cocycle for $S$ such that 
$1 \in L^\infty(\cA_0,\tau)$ 
is an analytic vector for the generator $d$ of the semigroup $E_0(u)$
acting by left-multiplication on $L^p(\cA_0,\tau)$. Then  
$$
b_t = L^p\text{-}\lim_{N\to \infty} \sum_{j=0}^{N-1}S_{jt/N}(u_{t/N}-1) 
$$
exists for any $1\le p <\infty$
and defines an additive cocycle $b$ for $S$ in 
$\bigcap_{p < \infty} L^p(\cA,\tau)$. 
\end{theorem}

Recently we could establish the most general case of stochastic differentiation: 
\begin{theorem}\label{dgeneral}
Let $u$ be a unitary cocycle for the shift $S$ and let 
$\eta\in \cD^{(\infty)}(d)\subset L^2(\cA_0,\tau)$.
Then  
$$
B_t(\eta) = L^2\text{-}\lim_{N\to \infty} \sum_{j=0}^{N-1}S_{jt/N}(u_{t/N}-1)\eta 
$$
exists for any $t>0$ and defines an additive $L^2$-cocycle $B(\eta)$ 
for any $\eta \in \cD^{(\infty)}(d)$.
%Moreover, $(B_t)_{t\ge 0}$ is a family of $\tau$-measureable operators
%in $L^0(\cA,\tau)$ which acts by left-multiplication on its domain.  
\end{theorem}
From this result we conjecture that the conditions in our presented main 
result can be relaxed such that it will include Stone's theorem for 
strongly continuous unitary groups on $L^2(\cA_0,\tau)$ whenever their generator
is a $\tau$-measureable 
operator. 

\specialsection{Structure Equations}\label{flucdiss}
\noindent
This section will shed some light on the background of the structure
equation $\[b,b\]+ b+b^*=0$ which appears in Theorem \ref{stone}.
As always in this survey, we assume the presence of an
$\cA_0$-valued white noise $(\cA,\tau, S,(\cA_I)_I)$.
Let us first focus on the situation that occurs in our extension of
Stone's Theorem.
\begin{theorem}\label{structure}
Let $u$ be a unitary cocycle for $S$ such that its associated 
semigroup $E_0(u)$ has the
generator $d \in L^\infty(\cA_0,\tau)$. Then the stochastic derivative $b$ of the
unitary cocycle $u$ is an additive cocycle in $L^2(\cA,\tau)$ which
enjoys the structure equation
$$\[b,b\]+b+b^* =0.$$
\end{theorem}

Investigations on Theorem \ref{structure} turned out that the
structure equation comes from a non-commutative
fluctuation-dissipation theorem for additive cocycles. 
Let us define for a given additive
$L^p$-cocycle $b$ the {\em flow of additive cocycles}
\begin{eqnarray*}
\beta\colon \cA_0 \times \bbR_0^+ \to L^2(\cA,\tau) \text{ by }
  \beta_t(x) = \[b,xb\]_t + xb_t + b^*_t x.
\end{eqnarray*}
Now the structure equation can be interpreted as follows.
The fluctuation part $ \[b,xb\]_t$ is balanced by the 
dissipation part $xb_t + b^*_tx$ for $x=1$.   
This interpretation is supported by the following result. We remind
that $d$ is the generator of the semigroup $E_0(u)$ and 
$L$ denotes the generator of the semigroup $E_0 (\operatorname{Ad}u)E_0$.

\begin{theorem}\label{flow}
Let $u$ be a unitary cocycle for $S$ and $b$ its stochastic
derivative as stated in \textup{Theorem \ref{structure}}. Then it holds
for any $x\in \cA_0$ and $t\ge 0$
\begin{eqnarray*}
\beta_t(x) &=& L^1\text{-}\lim_{N\to \infty}\sum_{j=0}^{N-1}
S_{jt/N}(u^*_{t/N}x u_{t/N}-x)\quad \text{ and }\\  
 L(x) &=& E_0(\beta_1(x)) = E_0(b_1^* x b_1) -d^* x d +
d^*x + x d.
\end{eqnarray*}
\end{theorem}
From this theorem one easily concludes the following fixed point properties.
\begin{corollary} 
For a projection $e \in \cA_0$ the following are equivalent:
\begin{itemize}
\item[(a)]
$e$ is a fixed point of the Markov process $(\cA, \tau,
(\operatorname{Ad}u) S, \cA_0)$
\item[(b)]
$e$ is a fixed point of the semigroup  $E_0
(\operatorname{Ad}u) E_0$. 
\item[(a')]
The flow $\beta(e)$ vanishes.
\item[(b')]
$L(e)$= 0 
\end{itemize}
\end{corollary}
Let us turn our attention to the case when the semigroup $E_0(u)$ has 
an unbounded generator $d$ as stated in Theorem \ref{dgeneral}. Looking
only at the semigroups $E_0(u)$ and $E_0(\operatorname{Ad}u) E_0$, 
it is not clear how the domains of the generators $d$ and $L$ are
related. But by passing to the level of flows and stochastic
derivatives, the domain of $L$ can be explicitly controlled by the
domain of $d$. Recall that $\cD^{(\infty)}(d) \subset L^\infty(\cA_0,\tau)$.   
\begin{theorem}
Let $u$ be a unitary cocycle for $S$ and let $\eta, \xi \in
\cD^{(\infty)}(d)$. Then
$$
\beta_t(\eta^*x \xi) = L^1\text{-}\lim_{N\to
\infty}\sum_{j=0}^{N-1}S_{jt/N}((u_{t/N}\eta)^*x u_{t/N}\xi -\eta^*x \xi) 
$$ 
exists for any $x \in \cA_0$ and $t \ge 0$. Moreover,
\begin{eqnarray*}
\beta_t(\eta^*x \xi)&=& (B_t(\eta))^* x \xi + \eta^* x B_t(\xi) +
\[B(\eta),x B(\xi)\]_t \quad \text{ and }\\
L(\eta^*x \xi) &=& E_0((B_1(\eta))^* x B_1(\xi)) - (d\eta)^*x d\xi + (d\eta)^*x \xi
+\eta^*x d\xi.  
\end{eqnarray*}
\end{theorem}
This theorem can be refined to give explicit information on the
domain of $L$ on the level of $L^p$-spaces for 
$1\le p < \infty$. Let us illustrate this:

\begin{corollary}
If $1 \in \cD^{(\infty)}(d)$ is an analytic vector for the generator
$d$ of the semigroup $E_0(u)$ then
$$
L(x) = E_0(b_1^* x b_1) -d^*x d + d^*x + x d  
$$
holds for any $x \in \cA_0$. Moreover, the generator $L$ of the
semigroup $E_0 (\operatorname{Ad} u) E_0$ maps $\bigcap_{p<\infty} L^p(\cA_0,\tau)$
into $\bigcap_{p < \infty} L^p(\cA_0,\tau)$.
\end{corollary}

\bibliographystyle{amsalpha}

\end{document}